\documentclass[10pt]{article}
\usepackage{amsmath,amsthm,amsfonts,amssymb,url,authblk}
\usepackage{tikz}
\usetikzlibrary{calc}
\usetikzlibrary{arrows,decorations.markings}

\newtheorem{theorem}{Theorem}[section]

\newtheorem{corollary}[theorem]{Corollary}

\newtheorem{example}{Example}[section]

\newcommand{\MTS}{\ensuremath{\mathsf{MTS}}} 
\newcommand{\STS}{\ensuremath{\mathsf{STS}}} 
\newcommand{\DTS}{\ensuremath{\mathsf{DTS}}} 
\newcommand{\zed}{\ensuremath{\mathbb{Z}}} 
\newcommand{\TT}{\ensuremath{\mathcal{T}}} 
\newcommand{\SSS}{\ensuremath{\mathcal{S}}} 
\newcommand{\RR}{\ensuremath{\mathcal{R}}} 
\newcommand{\D}{\ensuremath{\mathcal{D}}} 
\newcommand{\C}{\ensuremath{\mathcal{C}}} 
\newcommand{\M}[3]{\ensuremath{\mathcal{M}_{#1}#2#3}}

\title{Block-avoiding point sequencings of Mendelsohn triple systems}
\author[1]{Donald L.\ Kreher}
\author[2]{Douglas R.\ Stinson%
\thanks{D.R.\ Stinson's research is supported by  NSERC discovery grant RGPIN-03882.}}
\author[3]{Shannon Veitch}
\affil[1]{Department of Mathematical Sciences, 
Michigan Technological University 
Houghton, MI 49931,  
U.S.A.}
\affil[2]{David R.\ Cheriton School of Computer Science, University of Waterloo,
Waterloo, Ontario, N2L 3G1, Canada}
\affil[3]{Department of Combinatorics and Optimization, University of Waterloo,
Waterloo, Ontario, N2L 3G1, Canada}

\begin{document}
\maketitle

\begin{abstract}
A cyclic ordering of the points in a Mendelsohn triple system of order $v$ (or \MTS$(v)$) is called a \emph{sequencing}.
A sequencing $\D$ is \emph{$\ell$-good}
if there does not exist a triple $(x,y,z)$ in the \MTS$(v)$ such that \begin{enumerate}
\item  
the three points $x,y,$ and $z$  occur (cyclically) in that order in $\D$; and \item  
$\{x,y,z\}$ is a subset of $\ell$ cyclically consecutive points of $\D$.\end{enumerate}
In this paper, we prove some upper bounds on $\ell$ for \MTS$(v)$ having $\ell$-good sequencings and we prove that
any \MTS$(v)$ with $v \geq 7$ has a $3$-good sequencing. We also determine the optimal sequencings of every 
\MTS$(v)$ with $v \leq 10$. 
\end{abstract}

\section{Introducton}

There has been considerable recent interest in different kinds of block-avoiding sequencings of Steiner triple systems (or \STS$(v)$). See for example, \cite{Alspach,AKP,KS,KS2,SV}. A similar problem, in the setting of directed triple systems (or \DTS$(v)$), was introduced in \cite{KSV}. In this paper, we initiate a study of sequencings of Mendelsohn triple systems, or  \MTS$(v)$. 

A \emph{cyclic triple} is an ordered 
triple $(x,y,z)$, where $x,y,z$ are distinct. 
This triple contains the directed edges (or ordered pairs) $(x,y)$, $(y,z)$ and $(z,x)$
(we might also write these directed edges as $xy$, $yz$ and $zx$, respectively). 
Note that $(x,y,z)$, $(y,z,x)$ and $(z,x,y)$ are ``equivalent'' when considered as cyclic triples, i.e., they all contain the same three directed edges. The cyclic triple $(x,y,z)$ can be depicted as follows:

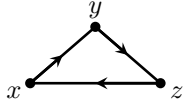
\begin{center}
\tikz[scale=0.5,line width=1]{
\coordinate (x) at (210:2.0);
\coordinate (y) at ( 90:0.5);
\coordinate (z) at (330:2.0);
\draw[->,>=stealth] (x)--($(x)!0.5!(y)$);
\draw[->,>=stealth] (y)--($(y)!0.5!(z)$);
\draw[->,>=stealth] (z)--($(z)!0.5!(x)$);
\draw(x)--(y)--(z)--cycle;
\node at ($(0,0)!1.25!(x)$) {$x$};
\node at ($(0,0)!1.90!(y)$) {$y$};
\node at ($(0,0)!1.25!(z)$) {$z$};
\foreach \i in {x,y,z}
{
\draw[fill] (\i) circle[radius=0.1];
}
}
\end{center}

Let $X$ be a set of $v$ points (or vertices) and let $\vec{K_v}$ denote the complete directed graph on vertex set $X$. This graph has $v(v-1)$ directed edges.
A  \emph{Mendelsohn triple system of order $v$} (see \cite{Men}) is a pair $(X, \TT)$, where 
$X$ is a set of $v$ \emph{points} and $\TT$ is a set of cyclic triples (or more simply, \emph{triples}) whose elements are members of $X$, such that every directed  edge in $\vec{K_v}$ occurs in exactly one triple in $\TT$. 
In graph-theoretic language, we are decomposing the complete directed graph into directed cycles of length three. 
%Of course, the triples in an \MTS$(v)$ are analogous to blocks in a Steiner triple system.
  
We will abbreviate the phrase ``Mendelsohn triple system of order $v$'' to 
\MTS$(v)$. It is well-known that an \MTS$(v)$ contains exactly $v(v-1)/3$ triples, and an 
\MTS$(v)$ exists if and only  if $v \equiv 0,1 \bmod 3$, $v \neq 6$. 
Various results on \MTS$(v)$ can  be found in \cite{CR}.

Suppose $(X, \TT)$ is an \MTS$(v)$, where, for convenience, 
$X = \{1, \dots , v\}$.
Suppose we arrange the points in $X$ in a directed cycle, say $\D = (i_1 \; i_2 \: \cdots \; i_v)$. 
We will refer to such a directed cycle as a \emph{sequencing}. Clearly, any cyclic shift of the sequencing $\D$ is equivalent to $\D$. 

A \emph{cyclic ordering} can be defined as a ternary relation as follows. 
Given a sequencing 
$\D = (i_1 \; i_2 \: \cdots \; i_v)$, we first define the associated total ordering $i_1 < i_2 < \cdots < i_v$.
Then we define the induced ternary relation $\C(\D)$ as follows 
\begin{center}
$[x,y,z] \in \C(\D)$ if and only if $x < y < z$ or $y < z < x$ or $z < x < y$.
\end{center}
Observe that any cyclic shift of $\D$ gives rise to the same ternary relation.

This definition can be explained informally as follows: In order to determine if a triple $[x,y,z] \in \C(\D)$, we start at $x$ and proceed around the directed cycle $\D$. Then $[x,y,z] \in \C(\D)$ if and only if we encounter $y$ before we encounter $z$. From this, it is obvious that exactly one of $[x,y,z]$ or $[x,z,y]$ is in $\C(\D)$.

We say that a cyclic triple $T = (x,y,z)$ is \emph{contained} in a sequencing $\D = (i_1 \; i_2 \: \cdots \; i_v)$
if $[x,y,z] \in \C(\D)$. 
For an integer $\ell \geq 3$, we say that the sequencing $\D$ is \emph{$\ell$-good}
if there does not exist a triple $(x,y,z) \in \TT$ such that 
\begin{enumerate}
\item $(x,y,z)$ is contained  in $\D$, and
\item $\{x,y,z\}$ is a subset of $\ell$ cyclically consecutive points of $\D$.
\end{enumerate}

Of course an $\ell$-good sequencing is automatically 
$\kappa$-good for all $\kappa$ such that $3 \leq \kappa \leq \ell-1$.

The basic questions we address in this paper are as follows:
\begin{itemize}
\item Given a particular \MTS$(v)$, say $(X,\TT)$, what is the largest integer $\ell$ such that $(X,\TT)$ has an
$\ell$-good sequencing?
\item Given a positive integer $v \equiv 0,1 \bmod 3$, $v \neq 6$, what is the largest integer $\ell$ such that \begin{itemize}
\item
there exists an \MTS$(v)$ that has an $\ell$-good sequencing, or 
\item every \MTS$(v)$  has an $\ell$-good sequencing?
\end{itemize}
\end{itemize}

\begin{example}
The triples $(0, 1, 3)$ and $(0, 3, 2)$, developed  modulo $7$, yield an  \MTS$(7)$.
It is not hard to see that 
$\D = (0 \; 1 \;  2  \; 3  \; 4  \; 5  \; 6)$ is a $3$-good sequencing for this \MTS$(7)$. This follows because:
\begin{enumerate}
\item
none of the seven triples obtained from $(0, 3, 2)$ are contained in $\D$, and
\item the seven triples obtained from $(0, 1,3)$ are contained in $\D$, but none of these triples 
is a subset of three cyclically consecutive points of $D$.
\end{enumerate}
However, this sequencing is not $4$-good, because each of the triples obtained from $(0,1,3)$ is a subset of four cyclically consecutive points of $\D$. 
\end{example}

The rest of this paper is organized as follows. In  Section \ref{nec.sec}, we prove an \MTS$(v)$ has an $\ell$-good sequencing only if $\ell \leq \lfloor \frac{v-1}{2} \rfloor$.  In Section \ref{small.sec},
we summarize the results of computer searches we used to determine the optimal sequencings of every 
\MTS$(v)$ with $v \leq 10$. 
In Section \ref{label.sec}, we prove that
any \MTS$(v)$ with $v \geq 7$ has a $3$-good sequencing.
Finally, in Section \ref{comments.sec}, we conclude with a few comments.

\section{Necessary Conditions}
\label{nec.sec}

\begin{theorem}
\label{even.thm}
Suppose $v$ is even. Then no \MTS$(v)$ has an $\ell$-good sequencing if $\ell \geq v/2$.
\end{theorem}

\begin{proof}
Without loss of generality, we assume the sequencing is \[\D = (0 \; 1 \;  \cdots \; v-1).\]
We will show that there is no \MTS$(v)$, say $(X,\TT)$, where $X = \{ 0, 1, \dots , v-1\}$ and for which $\D$ is
a $v/2$-good sequencing. In what follows, all arithmetic is modulo $v$.

There must be  a triple $(0,1,x) \in \TT$, where $x \in \{2, \dots , v-1\}$. If $2 \leq x < v/2$, then 
$\{0,1,x\}$ is a subset of the first $v/2$ points of $\D$, namely, $0, 1, \dots , v/2 - 1$. 
Similarly, if $2+v/2 \leq x \leq v - 1$, then 
$\{0,1,x\}$  a subset of $v/2$ cyclically consecutive points of $\D$, namely, $2+v/2,  \dots , v  - 1, 0 ,1$.
Hence, $x = v/2$ or $x = 1+v/2$ and thus either
\[R_0 = \left(0, 1, \frac{v}{2}\right) \quad \text{or} \quad S_0 = \left(0, 1, 1 + \frac{v}{2}\right)\] is a triple in $\TT$.

Similarly, it follows for each $i \in \zed_v$ that exactly one of 
\[R_i = \left(i, i+1, i+\frac{v}{2}\right) \quad \text{or} \quad S_i = \left(i, i+1, i+1 + \frac{v}{2}\right)\] is a triple in $\TT$. Let $\RR = \{ R_i : i \in \zed_{v}\}$ and let $\SSS = \{S_i : i \in \zed_{v}\}$.

Suppose $R_i \in \TT$. This triple contains the ordered pair $(i+v/2,i)$. The triple 
\[S_{i-1+v/2} =
\left(i-1+\frac{v}{2}, i+\frac{v}{2}, i\right)\] also contains the ordered pair $(i+v/2,i)$, so $S_{i-1+v/2} \not\in \TT$.
Then it must be the case that $R_{i-1+v/2} \in \TT$. Similarly, the triples $R_i$ and 
\[S_{i+v/2} =
\left(i+\frac{v}{2}, i+1+\frac{v}{2}, i+1\right)\] both contain the ordered pair $(i+1,i+v/2)$. Therefore  $S_{i+v/2} \not\in \TT$ and 
hence $R_{i+v/2} \in \TT$. In summary, if $R_i \in \TT$, then $R_{i-1+v/2} \in \TT$ and $R_{i+v/2} \in \TT$.

Now, using the fact that $R_{i-1+v/2} \in \TT$, we see that $R_{i-1+v/2+v/2} = R_{i-1} \in \TT$.  
Therefore, if $R_i \in \TT$, we have that $R_{i-1} \in \TT$.
From this, it follows easily that $\RR \subseteq \TT$ or $\SSS \subseteq \TT$. We consider the following two cases.

\medskip

\noindent\textbf{Case 1 }: $\RR \subseteq \TT$ and $\SSS \cap \TT = \emptyset$.

\smallskip

The triples in $\RR$ cover all ordered pairs having differences $1$, $v/2 - 1$ and $v/2$, where the 
\emph{difference} of a pair $(a,b)$ is $(b - a) \bmod v$.
Now consider the ordered pair $(0,2)$, which has difference $2$. There must be a triple $(0,2,x) \in \TT$.

If $3 \leq x  < v/2$, then 
$\{0,2,x\}$ is a subset of the first $v/2$ points of $\D$, namely, $0, 1, \dots , v/2 - 1$. 
Similarly, if $3+v/2 \leq x \leq v - 1$, then 
$\{0,2,x\}$ is again a subset of $v/2$ cyclically consecutive points of $\D$, 
namely, $3+v/2,  \dots , v - 1, 0 ,1,2$.

If $x \in \{v/2, 1+v/2,2+v/2\}$, then we have have two pairs with difference $v/2-1$ or $v/2$, because
\begin{eqnarray*}
0 - \frac{v}{2} &=& \frac{v}{2},\\
1 + \frac{v}{2} - 2 &=& \frac{v}{2} - 1, \quad \text{ and}\\
  2+ \frac{v}{2}  - 2 &=& \frac{v}{2}.
  \end{eqnarray*}
Therefore, $x = 1$. 

It follows in a similar manner that all the triples of the form $(i,i+2,i+1)$ are in $\TT$. But this is impossible because $(0,2,1)$ and $(1,3,2)$ both contain the ordered pair $(2,1)$.

\medskip

\noindent\textbf{Case 2 }: $\SSS \subseteq \TT$ and $\RR \cap \TT = \emptyset$.

\smallskip

The triples in $\SSS$ also cover all ordered pairs having differences $1$, $v/2-1$ and $v/2$.
Therefore the proof is identical to case 1.

\end{proof}

Now we turn to the case of odd $v$.

\begin{theorem}
\label{odd.thm}
Suppose $v$ is odd. Then no MTS$(v)$ has an $\ell$-good sequencing if $\ell \geq (v+1)/2$.
\end{theorem}

\begin{proof} 
The proof is similar to that of Theorem \ref{even.thm}.
We assume the sequencing is $\D = (0 \; 1 \;  \cdots \; v-1)$.
We will show that there is no \MTS$(v)$, say $(X,\TT)$, where $X = \{ 0, 1, \dots , v-1\}$ and for which $\D$ is
a $(v+1)/2$-good sequencing. 
There must be  a triple $(0,1,x) \in \TT$, where $x \in \{2, \dots , v-1\}$. If $2 \leq x \leq (v-1)/2$, then 
$\{0,1,x\}$ is a subset of the first $(v+1)/2$ points of $\D$, namely, $0, 1, \dots , (v-1)/2$. 
Similarly, if $(v+3)/2 \leq x \leq v - 1$, then 
$\{0,1,x\}$ is a also a subset of   $(v+1)/2$ cyclically consecutive points of $\D$, 
namely, $(v+3)/2,  \dots , v  - 1, 0 ,1$. Hence, it must be the case that $x = (v+1)/2$, i.e., $(0,1,(v+1)/2) \in \TT$.

An identical argument shows that $\TT$ must contain all of the triples 
\[\left(i,i+1,i+\frac{v+1}{2}\right),\] where arithmetic is moduli  $v$ and $0 \leq i \leq v-1$. In particular, $\TT$ contains the triples 
\[ \left(0,1,\frac{v+1}{2}\right) \quad \text{and} \quad
\left(\frac{v+1}{2}, \frac{v+3}{2}, 1\right).\] But these two triples both contain the ordered pair $(1,(v+1)/2)$, so we have a contradiction.
\end{proof}

Combining Theorems \ref{even.thm} and \ref{odd.thm}, we obtain the following.

\begin{corollary}
\label{bound.cor}
If an \MTS$(v)$ has an $\ell$-good sequencing, then $\ell \leq \lfloor \frac{v-1}{2} \rfloor$.
\end{corollary}

\section{Sequencings of \MTS$(v)$ for Small Values of $v$}
\label{small.sec}

We have determined the optimal sequencings for all \MTS$(v)$ with $v \leq 10$.
The results are given in Table \ref{summary}. This table lists the number of nonisomorphic \MTS$(v)$ for each $v$,
along with the number of designs that have $3$-good and $4$-good sequencings.  None of these designs have $5$-good sequencings, by Corollary \ref{bound.cor}. 

We present the three \MTS$(9)$ that have $4$-good sequencings, as well as  the five \MTS$(10)$ that do not have have $4$-good sequencings, in the Appendices. 
Additional details can be found in the technical report
\cite{KSV2}. 

We noticed one particularly interesting fact concerning the five nonisomorphic \MTS$(10)$ that do not have a 4-good sequencing. If any triple is removed from one of these five \MTS$(10)$, then the resulting ``partial'' 
\MTS$(10)$ having $29$ triples turns out to have a $4$-good sequencing (we verified this fact by computer). So these \MTS$(10)$ ``almost'' have $4$-good sequencings. In fact, we know from these results that any ``partial'' 
\MTS$(10)$ having $29$ triples has a $4$-good sequencing. This is because such a partial
\MTS$(10)$ can automatically be completed to an \MTS$(10)$, and therefore any partial
\MTS$(10)$ having $29$ triples arises from the deletion of a triple from an \MTS$(10)$. Clearly, if we delete a triple from an \MTS$(10)$ that has a $4$-good sequencing, then the resulting partial
\MTS$(10)$ also has a $4$-good sequencing.

\begin{table}[tb]
\caption{Sequencings of \MTS$(v)$ with $v \leq 10$}\label{summary}
\begin{center}
\begin{tabular}{*{6}{r}}
\hline
  & 
  Nonisomorphic& 
  \multicolumn{2}{c}{$\ell$-good sequencings}\\
  $v$&
  $\MTS(v)$&
  $\ell=3$ & $\ell=4$ \\
\hline
  3& 1&  0&  0  \\
  4&  1&  0&  0  \\
  6&   0&  0&  0\\
  7&   3&  3&  0  \\
  9& 18& 18&  3  \\
 10&143&143&138 \\
 \hline
\end{tabular}
\end{center}
\end{table}

\section{Constructing 3-good Sequencings}
\label{label.sec}

%In this section, we investigate a possible method for obtaining 3-good sequencings of \MTS$(v)$. 
Charlie Colbourn proved that any \STS$(v)$ has a 3-good sequencing. His method is described in \cite{KS}; it is based on examining the triples that contain a particular point $x$ and then relabelling the points in a suitable way. We have adapted this approach to obtain 3-good sequencings of \MTS$(v)$; however, it turned out to be quite a bit more complicated to obtain the desired result for \MTS$(v)$ that it did for \STS$(v)$.

Suppose $(X,\TT)$ is an \MTS$(v)$ and fix a particular point $x \in X$. 
Construct a directed graph $G_x$ on vertex set
$X \setminus \{x\}$ as follows. For every triple $(x,y,z) \in \TT$ (or a cyclic rotation of this triple), 
include the directed edge $(y,z)$ in $G_x$.
It is not hard to see that $G_x$ consists of a vertex-disjoint union of one or more directed cycles (note that some of these directed cycles could have length two).  Suppose the directed cycles are named $C_1, C_2, \dots ,C_s$.
We can construct a (cyclic) sequencing $\D_x$ of $X \setminus \{x\}$ by writing out the cycles 
$C_1, C_2, \dots ,C_s$ in order. For each of the cycles $C_i$, we can arbitrarily pick any vertex in the cycle as a starting point. 

It is easy to see that no three cyclically consecutive vertices of the sequencing $\D_x$ comprise a triple. 
Consider three consecutive vertices, say  $x_i$, $x_j$ and $x_k$. 
At least one of  $(x_i,x_j)$ or $(x_j,x_k)$ is an edge in $G_x$. 
In the first case, $(x,x_i,x_j) \in \TT$ so $(x_i,x_j,x_k) \not\in \TT$, and in the second case,
$(x,x_j,x_k) \in \TT$ so again $(x_i,x_j,x_k) \not\in \TT$.

The difficulty is that, if we insert $x$ into $\D_x$ in any position, the sequencing is no longer $3$-good. So we need to modify $\D_x$ at the same time that we insert $x$. We  illustrate how this can be done, in various situations, in the rest of this section.

\subsection{A Directed Cycle of Length at least Six}
\label{t.sec}

Let's suppose that $C_1$ is a directed cycle of length $\tau \geq 6$, say $(1 , 2 , \dots , \tau)$. This means that the following five triples are in $\TT$:
\[
(x,1,2)  \quad (x,2,3) \quad (x,3,4) \quad (x,4,5) \quad (x,5,6).
\]
Suppose that $\D_x = (1 \; 2 \; \cdots \; \tau \; \cdots \; v)$.  Replace the four  vertices 
$1 \; 2 \; 3 \;  4$ by $3 \; 4 \; 1 \; x \; 2$, obtaining a sequencing $\D$ of $X$. Notice that $v$ is the vertex preceding $3$ in $\D$. We check that there are no triples comprising three consecutive vertices of the modified sequencing $\D$:
\begin{center}
\begin{tabular}{ll}
$(v,3,4)$ $\not\in \TT$ &because $(x,3,4)$ $\in \TT$\\
$(3,4,1)$ $\not\in \TT$ &because $(x,3,4)$ $\in \TT$\\
$(4,1,x)$ $\not\in \TT$ &because $(x,4,5)$ $\in \TT$\\
$(1,x,2)$ $\not\in \TT$ &because $(x,2,3)$ $\in \TT$\\
$(x,2,5)$ $\not\in \TT$ &because $(x,2,3)$ $\in \TT$\\
$(2,5,6)$ $\not\in \TT$ &because $(x,5,6)$ $\in \TT$.
\end{tabular}
\end{center}

\subsection{Two Directed Cycles, Each of Length at Least Three}
\label{3.3.sec}

Suppose that $G_x$ contains two directed cycles of length at least three, say $(y,3,4,5, \dots)$ (note that it is possible that $y= 5$, if this cycle has length three)
and $(1,2,z,\dots)$. We can assume  
that $y \; 3 \; 4\; 1\; 2\; z$ are consecutive vertices in $\D_x$.

The following five triples are in $\TT$:
\[
(x,1,2) \quad (x,2,z) \quad  (x,y,3) \quad  (x,3,4) \quad  (x,4,5).
\]

Delete the two consecutive vertices $4\; 1$ from $\D_x$ and replace them by $1 \; x \; 4$, obtaining a sequencing $\D$ of $X$. We check that there are no triples comprising three consecutive vertices of the modified sequencing $\D$:
\begin{center}
\begin{tabular}{ll}
$(y,3,1)$ $\not\in \TT$ &because $(y,3,x)$ $\in \TT$\\
$(3,1,x)$ $\not\in \TT$ &because $(x,3,4)$ $\in \TT$\\
$(1,x,4)$ $\not\in \TT$ &because $(x,4,5)$ $\in \TT$\\
$(x,4,2)$ $\not\in \TT$ &because $(x,4,5)$ $\in \TT$\\
$(4,2,z)$ $\not\in \TT$ &because $(x,2,z)$ $\in \TT$.
\end{tabular}
\end{center}

\subsection{Two Directed Cycles of Length Two}
\label{2.2.sec}

Suppose that $v \geq 7$ and there are two directed cycles in $G_x$ having length two,  
say $(1 , 2)$ and $(3 , 4)$.
We assume that $3 \; 4 \; 1\; 2$ are consecutive vertices in $\D_x$.
The following four triples are in $\TT$:
\[
(x,1,2) \quad  (x,2,1) \quad  (x,3,4) \quad  (x,4,3).
\]
 
Choose $z \geq 5$ such that $(4,2,z) \not\in \TT$ and 
choose $y \geq 5$ such that $(3,1,y) \not\in \TT$ and $y \neq z$.
We also require that
\begin{itemize}
\item $y$ and $z$ are in different directed cycles in $G_x$, or
\item if $G_x$ contains only three directed cycles, then $y$ immediately precedes $z$ in a directed cycle in $G_x$
(this can be done if $v \geq 7$).
\end{itemize}
Then we can assume that $y$ immediately precedes $3$ in $\D_x$ and $z$ immediately follows $2$ in $\D_x$. 
Now, delete the two consecutive vertices  $1\; 4$ from $\D_x$ and replace them by $4\; 1 \; x$, obtaining a sequencing $\D$ of $X$. 

We check that there are no triples comprising three consecutive vertices of the modified sequencing $\D$:
\begin{center}
\begin{tabular}{ll}
$(y,3,1)$ $\not\in \TT$ &by the choice of $y$\\
$(3,1,x)$ $\not\in \TT$ &because $(x,3,4)$ $\in \TT$\\
$(1,x,4)$ $\not\in \TT$ &because $(x,4,3)$ $\in \TT$\\
$(x,4,2)$ $\not\in \TT$ &because $(x,4,3)$ $\in \TT$\\
$(4,2,z)$ $\not\in \TT$  &by the choice of $z$.
\end{tabular}
\end{center}

\subsection{Two Directed Cycles, One of Length Two}
\label{2.t.sec}

Suppose that $v \geq 6$ and $G_x$ consists of exactly two directed cycles, 
one of length two, say $(1 , 2)$, and one of length $v-2$,
say $(3,4 , \dots , v)$. This implies that the following  triples are in $\TT$:
\[
(x,1,2)  \quad (x,2,1) \quad (x,3,4) \quad (x,4,5) \quad 
 (x,5,6) \quad (x,v-1,v) \quad (x,v,3).
\]
We assume that $\D_x =  (1\; 2\; 3 \; 4 \; \cdots \; v-1 \; v)$.

Now, if $(5,4,2) \in \TT$, then $(5,4,1) \not\in \TT$. Therefore by interchanging $1$ and $2$ if necessary, we can assume that $(5,4,2) \not\in \TT$.

Replace the four  vertices 
$1 \; 2 \; 3 \;  4$ in $\D_x$ by $3 \; 1 \; x \; 4 \; 2$ to construct the sequencing $\D$, so 
We check that there are no triples comprising three consecutive vertices of the  sequencing $\D$:
\begin{center}
\begin{tabular}{ll}
$(v-1,v,3)$ $\not\in \TT$ &because $(x,v-1,v)$ $\in \TT$\\
$(v,3,1)$ $\not\in \TT$ &because $(x,v,3)$ $\in \TT$\\
$(3,1,x)$ $\not\in \TT$ &because $(x,3,4)$ $\in \TT$\\
$(1,x,4)$ $\not\in \TT$ &because $(x,4,5)$ $\in \TT$\\
$(x,4,2)$ $\not\in \TT$ &because $(x,4,5)$ $\in \TT$\\
$(4,2,5)$ $\not\in \TT$ &by assumption\\
$(2,5,6)$ $\not\in \TT$ &because $(x,5,6)$ $\in \TT$.
\end{tabular}
\end{center}

\subsection{The Main Theorem}

We can now show that the four cases we have considered cover all possibilities. 
Suppose that $v \geq 7$ and we classify $G_x$ according to the number of directed cycles of length two that it contains.
\begin{itemize}
\item If $G_x$ has at least two directed cycles of length two, use the construction in Section \ref{2.2.sec}.
\item If $G_x$ has exactly one directed cycle of length two, then either
\begin{itemize}
\item $G_x$ contains at least two directed cycles of length at least three, in which case we can 
use the construction in Section \ref{3.3.sec}, 
or 
\item $G_x$ consists of exactly two directed cycles, one of length two and one of length at least five, so we 
can use the construction in Section \ref{2.t.sec}.
\end{itemize} 
\item If $G_x$ has no  directed cycles of length two, then either
\begin{itemize}
\item $G_x$ contains at least two directed cycles of length at least three, in which case we can 
use the construction in Section \ref{3.3.sec}, 
or 
\item $G_x$ consists of a single directed cycle of  length at least seven, so we 
can use the construction in Section \ref{t.sec}.
\end{itemize} 
\end{itemize} 

Therefore, we have the following result.

\begin{theorem} Any \MTS$(v)$ with $v \geq 7$ has a $3$-good sequencing.
\end{theorem}

\section{Comments}
\label{comments.sec}

Recent papers  have considered $\ell$-good sequencings for \STS$(v)$, \DTS$(v)$ and \MTS$(v)$ (however, we should note that the definition of $\ell$-good sequencing is slightly different in each case).
It is interesting to compare the results obtained for these three types of triple systems.
\begin{itemize}
\item For  \STS$(v)$, \cite{SV} establishes that an $\ell$-good sequencing exists only if $\ell \leq (v + 2)/3$.
But there are only a few small examples known where this bound is met with equality. In fact, it is currently unknown if there is an infinite class of \STS$(v)$ that have $(cv)$-good sequencings, for any positive constant $c$.  Proving this for $c \approx 1/2$ would be the best possible result in light of current knowledge, but it would still be of interest if we could establish this result for some smaller value of $c$, say $c = 1/4$. 
It is also known that every \STS$(v)$ has a $3$-good sequencing (see \cite{KS}); every  \STS$(v)$ with $v> 71$ has a $4$-good sequencing (see \cite{KS}); and every \STS$(v)$  with $v \geq \ell^6/16$ has an $\ell$-good sequencing (see \cite{SV}).

%We briefly considered  recursive constructions, but all  we have so far is Shannon's result that a doubling construction for \STS$(v)$ allows the value of $\ell$ to be increased by one. 
\item For \DTS$(v)$ (i.e., directed triple systems of order $v$), it is possible that a $v$-good sequencing exists. 
In fact, there is a \DTS$(v)$ having a $v$-good sequencing for all permissible values of $v$ (see \cite{KSV}).
It is also shown in \cite{KSV} that there is a \DTS$(v)$ that does not have a $v$-good sequencing, for all 
$v \equiv 0,1 \bmod 3$, $v \geq 7$. 
\item In this paper, we showed that
$\ell \leq \lfloor \frac{v-1}{2} \rfloor$ is a necessary condition for the existence of an
$\ell$-good sequencing of an \MTS$(v)$. We showed  that 
this bound is met with equality for $v = 7,9,10$ and we  also proved that every \MTS$(v)$ has a $3$-good sequencing. 
\end{itemize}

%\newpage

\appendix

\section{The three \MTS$(9)$ that have $4$-good sequencings}

\begin{description}
\item[\M{9}{1}{.1}]	
\end{description}

$\begin{array}{@{}*{7}{l@{\hspace{4pt}}}l@{}}
 (0,2,1)& (0,1,6)& (0,3,2)& (0,7,3)& (0,4,7)& (0,8,4)& (0,6,5)& (0,5,8)\\
 (1,2,7)& (1,3,6)& (1,8,3)& (1,5,4)& (1,4,8)& (1,7,5)& (2,3,8)& (2,4,6)\\
 (2,6,4)& (2,5,7)& (2,8,5)& (3,4,5)& (3,7,4)& (3,5,6)& (6,7,8)& (6,8,7)
\end{array}$

\medskip

Lexicographic least 4-good sequencing : 023471856

\medskip

Number of 4-good sequencings found: 18

\begin{description}
\item[\M{9}{3}{.1}]
\end{description}
	
$\begin{array}{@{}*{7}{l@{\hspace{4pt}}}l@{}}
 (0,2,1)& (0,1,3)& (0,6,2)& (0,3,8)& (0,4,6)& (0,7,4)& (0,5,7)& (0,8,5)\\
 (1,2,7)& (1,8,3)& (1,6,4)& (1,4,8)& (1,5,6)& (1,7,5)& (2,6,3)& (2,3,7)\\
 (2,4,5)& (2,8,4)& (2,5,8)& (3,5,4)& (3,4,7)& (3,6,5)& (6,7,8)& (6,8,7)
\end{array}$

\medskip

Lexicographic least 4-good sequencing : 047563812

\medskip

Number of 4-good sequencings found: 36

\begin{description}
\item[\M{9}{7}{.1}]	
\end{description}

$\begin{array}{@{}*{7}{l@{\hspace{4pt}}}l@{}}
 (0,1,2)& (0,2,1)& (0,6,3)& (0,3,8)& (0,4,6)& (0,7,4)& (0,5,7)& (0,8,5)\\
 (1,3,6)& (1,7,3)& (1,4,7)& (1,8,4)& (1,6,5)& (1,5,8)& (2,3,7)& (2,8,3)\\
 (2,6,4)& (2,4,8)& (2,5,6)& (2,7,5)& (3,4,5)& (3,5,4)& (6,7,8)& (6,8,7)
\end{array}$

\medskip

	Lexicographic least 4-good sequencing : 031485726

\medskip

	Number of 4-good sequencings found: 324

\section{The five \MTS$(10)$ that do not have $4$-good sequencings}

\begin{description}
\item[\M{10}{116}{.1}]
\end{description}

$\begin{array}{@{}*{5}{l@{\hspace{4pt}}}l@{}}
 (0,1,8)& (0,9,1)& (0,5,2)& (0,2,7)& (0,4,3)& (0,3,6)\\
 (0,7,4)& (0,6,5)& (0,8,9)& (1,2,3)& (1,3,2)& (1,4,5)\\
 (1,5,4)& (1,6,7)& (1,7,6)& (1,9,8)& (2,6,4)& (2,4,8)\\
 (2,5,9)& (2,8,6)& (2,9,7)& (3,4,9)& (3,7,5)& (3,5,8)\\
 (3,9,6)& (3,8,7)& (4,6,8)& (4,7,9)& (5,6,9)& (5,7,8)
\end{array}$

\begin{description}
\item[\M{10}{116}{.2}]
\end{description}

$\begin{array}{@{}*{5}{l@{\hspace{4pt}}}l@{}}
 (0,1,8)& (0,9,1)& (0,5,2)& (0,2,7)& (0,4,3)& (0,3,6)\\
 (0,7,4)& (0,6,5)& (0,8,9)& (1,2,3)& (1,3,2)& (1,4,5)\\
 (1,5,4)& (1,6,7)& (1,7,6)& (1,9,8)& (2,6,4)& (2,4,8)\\
 (2,5,9)& (2,8,6)& (2,9,7)& (3,4,9)& (3,5,7)& (3,8,5)\\
 (3,9,6)& (3,7,8)& (4,6,8)& (4,7,9)& (5,6,9)& (5,8,7)
\end{array}$

\begin{description}
\item[\M{10}{118}{.1}]
\end{description}

$\begin{array}{@{}*{5}{l@{\hspace{4pt}}}l@{}}
 (0,1,8)& (0,9,1)& (0,4,2)& (0,2,7)& (0,5,3)& (0,3,6)\\
 (0,7,4)& (0,6,5)& (0,8,9)& (1,2,3)& (1,3,2)& (1,4,5)\\
 (1,5,4)& (1,6,7)& (1,7,6)& (1,9,8)& (2,4,6)& (2,8,5)\\
 (2,5,9)& (2,6,8)& (2,9,7)& (3,8,4)& (3,4,9)& (3,5,7)\\
 (3,9,6)& (3,7,8)& (4,8,6)& (4,7,9)& (5,6,9)& (5,8,7)
\end{array}$

\begin{description}
\item[\M{10}{134}{.1}]
\end{description}

$\begin{array}{@{}*{5}{l@{\hspace{4pt}}}l@{}}
 (0,1,8)& (0,9,1)& (0,5,2)& (0,2,7)& (0,4,3)& (0,3,6)\\
 (0,6,4)& (0,7,5)& (0,8,9)& (1,2,3)& (1,3,2)& (1,4,5)\\
 (1,5,4)& (1,6,7)& (1,7,6)& (1,9,8)& (2,4,8)& (2,8,4)\\
 (2,5,7)& (2,6,9)& (2,9,6)& (3,4,6)& (3,5,9)& (3,9,5)\\
 (3,7,8)& (3,8,7)& (4,7,9)& (4,9,7)& (5,6,8)& (5,8,6)
\end{array}$

\begin{description}
\item[\M{10}{134}{.2}]
\end{description}

$\begin{array}{@{}*{5}{l@{\hspace{4pt}}}l@{}}
 (0,1,8)& (0,9,1)& (0,2,5)& (0,7,2)& (0,4,3)& (0,3,6)\\
 (0,6,4)& (0,5,7)& (0,8,9)& (1,2,3)& (1,3,2)& (1,4,5)\\
 (1,5,4)& (1,6,7)& (1,7,6)& (1,9,8)& (2,4,8)& (2,8,4)\\
 (2,7,5)& (2,6,9)& (2,9,6)& (3,4,6)& (3,5,9)& (3,9,5)\\
 (3,7,8)& (3,8,7)& (4,7,9)& (4,9,7)& (5,6,8)& (5,8,6)
\end{array}$

\end{document}